\documentclass[10pt]{amsart}
\numberwithin{equation}{section}
\usepackage{latexsym}
\usepackage{amsmath}
\usepackage{amssymb}
\usepackage{mathrsfs}
\usepackage{graphicx}
\usepackage{color}
\usepackage{pgfpages}
\usepackage{ifthen}
\usepackage{leftidx,tensor}
\usepackage[T1]{fontenc}
\usepackage[latin1]{inputenc}
\usepackage{mathtools}
\usepackage{comment}
\usepackage{nicefrac}

\numberwithin{equation}{section}

\newtheorem{defi}{Definition}[section]
\newtheorem{theorem}[defi]{Theorem}
\newtheorem{lemma}[defi]{Lemma}
\newtheorem{corollary}[defi]{Corollary}
\newtheorem{proposition}[defi]{Proposition}
\newtheorem{remark}[defi]{Remark}

\newtheorem{innerassumption}{Assumption}
\newenvironment{assumption}[1]
  {\renewcommand\theinnerassumption{#1}\innerassumption}
  {\endinnerassumption}

\newcommand{\be}{\begin{equation} \label}
\newcommand{\ee}{\end{equation}}
\newcommand{\bea}{\begin{eqnarray}\label}
\newcommand{\eea}{\end{eqnarray}}
\newcommand{\bas}{\begin{eqnarray*}}
\newcommand{\eas}{\end{eqnarray*}}
\newcommand{\bit}{\begin{itemize}}
\newcommand{\eit}{\end{itemize}}
\newcommand{\R}{\mathbb{R}}

\newcommand{\cA}{\mathcal{A}}

\newcommand{\F}{\mathbb{F}}
\newcommand{\E}{\mathbb{E}}
\newcommand{\B}{\mathbb{B}}

\newcommand{\eps}{\varepsilon}

\newcommand{\hra}{\hookrightarrow}

\newcommand{\ccor}{c_{\mbox{\tiny{cor}}}}
\newcommand{\Catm}{C_{\mbox{\tiny{atm}}}}
\newcommand{\Cocean}{C_{\mbox{\tiny{ocean}}}}
\newcommand{\Uatm}{U_{\mbox{\tiny{atm}}}}
\newcommand{\Uocean}{U_{\mbox{\tiny{ocean}}}}
\newcommand{\Ratm}{R_{\mbox{\tiny{atm}}}}
\newcommand{\Rocean}{R_{\mbox{\tiny{ocean}}}}
\newcommand{\ratm}{\rho_{\mbox{\tiny{atm}}}}
\newcommand{\rice}{\rho_{\mbox{\tiny{ice}}}}
\newcommand{\rocean}{\rho_{\mbox{\tiny{ocean}}}}
\newcommand{\tatm}{\tau_{\mbox{\tiny{atm}}}}
\newcommand{\tocean}{\tau_{\mbox{\tiny{ocean}}}}

\newcommand{\vbar}{\overline{v}}
\newcommand{\normf}{\| f \|_{C^1}}
\newcommand{\normftime}{\| f \|_{C^1(J \times \R_+)}}

\DeclareMathOperator*{\divergence}{div}

\DeclareMathOperator*{\trace}{tr}

\allowdisplaybreaks
\begin{document}

\title[ Strong Time Periodic Solutions to Hibler's Sea Ice Model]
{Time Periodic Solutions to Hibler's Sea Ice Model}

\author{Felix Brandt}
\address{Technische Universit\"at Darmstadt\\
        Fachbereich Mathematik\\
        Schlossgartenstrasse 7\\
        64289 Darmstadt, Germany}
\email{brandt@mathematik.tu-darmstadt.de}

\author{Matthias Hieber}
\address{Technische Universit\"at Darmstadt\\
        Fachbereich Mathematik\\
        Schlossgartenstrasse 7\\
        64289 Darmstadt, Germany}
\email{hieber@mathematik.tu-darmstadt.de}


\begin{abstract}
It is shown that the viscous-plastic Hibler sea ice model admits a unique, strong $T$-time periodic solution provided the given $T$-periodic forcing functions are small in suitable norms. 
This is in particular true for time periodic wind forces and time periodic ice growth rates.
\end{abstract}

\maketitle

\section{Introduction}\label{secintro}
Sea ice is a material with a complex mechanical and thermodynamical behaviour. 
Freezing sea water forms a composite of pure ice, liquid brine, air pockets and solid salt. 
The details of this formation depend on the laminar or turbulent environmental conditions, see e.g.\ \cite{Fel08} and \cite{Gol15}. 
This composite responds differently to heating, pressure or mechanical forces than for example the (salt-free) glacial ice of ice sheets. 
For a recent survey in the Notices of the AMS, see \cite{Gol20}. 

The governing equations of large-scale sea ice dynamics that form the basis of many sea ice models in climate science were suggested in a seminal paper by Hibler \cite{Hib79} in 1979. 
Sea ice is here modeled as a material with a very specific constitutive law combined with viscous-plastic rheology. 

This set of equations has been investigated numerically by various communities during the last decades, see e.g.\ \cite{Mehalle21,MK21,Meh19,MR17,SK18,DWT15,KDL15,LT09}. 
Somewhat surprisingly, and in contrast to the equations describing atmospheric or oceanic dynamics as e.g.\ the primitive equations, rigorous analysis of the sea ice equations started only very recently by the work of Brandt, Disser, Haller-Dintelmann, Hieber \cite{BDHH22} and Liu, Thomas and Titi \cite{LTT21}. 
This might be due to the  fact that the underlying set of equations is a coupled degenerate quasilinear parabolic-hyperbolic system, whose analysis is delicate.    
 
In \cite{BDHH22}, it was shown by means of quasilinear evolution equations that a suitable regularization of Hibler's model, coupling velocity, thickness and compactness of sea ice, is 
locally strongly well-posed and also globally strongly well-posed for initial data close to constant equilibria. 
A different  approach developed  in \cite{LTT21} emphasizes the hyperbolic-parabolic character of Hibler's model and proves also local strong well-posedness by means of a 
different regularization and by energy estimates. 

A different model describing the evolution of the thickness of shallow ice sheets was proposed and studied very recently by Piersanti and Temam \cite{PT22}.   

In this article, we focus on {\em periodic solutions} to Hibler's sea ice model and in particular on {\em periodic wind forces} and {\em periodic ice growth rates}.  
We prove, roughly speaking, that the sea ice system described precisely in \eqref{eq:cs} below admits a unique $T$-periodic solution in a small neighborhood of an equilibrium point of the system 
subject to certain smallness assumptions. These assumptions are in particular satisfied for sufficiently small $T$-periodic wind forces of a certain shape as well as for  $T$-periodic ice growth rates.

Time periodic problems have a long tradition, for example in fluid mechanics. 
For results in this direction, see e.g.\ \cite{GHH16}, \cite{KMT14}, \cite{Gal13}, \cite{GS04}.
For a  survey of various existing approaches in this context, we refer to the article \cite{GK18} by Galdi and Kyed.

Note, however, that in contrast to the Navier-Stokes equations, where the underlying stress tensor is linear,  the stress tensor associated to Hibler's  sea ice model is of  {\em quasilinear nature}. Hence, 
different and  new methods need to be developed. We present here an  approach which relies on a quasilinear version of the Arendt-Bu theorem \cite{AB02} on maximal periodic $L^p$-regularity 
developed in \cite{HS20}. In our main result, we show the existence and uniqueness of time periodic solutions in a suitable  neighborhood of constant equilibria. In particular,  we obtain 
periodic solutions to Hibler's sea ice system for periodic wind forces and periodic ice growth rates. 

This paper is organized as follows. 
Section~\ref{sec:mainresults} presents some preliminaries as well as our main results on the unique existence of time periodic strong solutions to Hibler's sea ice model.  
In Section~\ref{sec:maxperiodicreg}, we discuss maximal periodic $L^p$-regularity, while Section~\ref{sec:proofprepresults} is dedicated to regularity as well as Lipschitz continuity of the 
quasilinear and forcing terms. Finally, the proof of our main result is given in Section~\ref{sec:proofmainresults}.

\section{Preliminaries and Main Results}\label{sec:mainresults}
The aim of this section is twofold: After describing and recalling Hibler's sea ice model in some detail, we introduce our setting for periodic solutions and  present then 
our main result as well as the aforementioned corollaries to time periodic atmospheric wind forces and ice growth rates.
      
In Hibler's model, the momentum balance is given by the two-dimensional equation     
\begin{equation}\label{eqmomentumbalance}
    m(\dot{u} + u \cdot \nabla u) = \divergence \sigma - m \ccor n \times u - m g \nabla H + \tatm + \tocean, 
\end{equation}     
where ${u\colon (0,\infty) \times \R^2 \to \R^2}$ represents the horizontal ice velocity and $m$ the ice mass per unit area. 
Moreover, ${- m \ccor \, n \times u}$ is the Coriolis force with Coriolis parameter $\ccor > 0$, while ${n\colon \R^2 \to \R^3}$ denotes the unit vector normal to the surface, and $- m g \nabla H$ describes the force arising from changing sea surface tilt, where $g$ represents gravity and $H\colon (0,\infty) \times \R^2 \to [0,\infty)$ denotes the sea surface dynamic height.
The terms $\tatm$ and $\tocean$ represent atmospheric wind and oceanic forces. 
They are given by
\begin{equation*}
\tatm = \ratm \Catm \vert \Uatm \vert \Ratm \Uatm \quad \mbox{and} \quad \tocean = \rocean \Cocean \vert \Uocean - u \vert \Rocean (\Uocean - u),
\end{equation*}
where $\Uatm$ and $\Uocean$ are the surface winds of the atmosphere and the surface velocity of the ocean, respectively.  
Furthermore, $\Catm$ and $\Cocean$ denote air and ocean drag coefficients, $\ratm$ and $\rocean$ represent the densities for air and sea water, and $\Ratm$ and $\Rocean$ are rotation matrices acting on wind and current vectors.

Following Hibler \cite{Hib79}, the constitutive law for the ice stress is given by
\begin{equation*}
    \sigma  = 2 \eta(\eps,P) \eps + [\zeta(\eps,P) - \eta(\eps,P)]\trace(\eps)I - \frac{P}{2}I,
\end{equation*}
where $\eps = \eps(u) = \frac{1}{2}\left((\nabla u + (\nabla u)^T\right)$ is the deformation tensor, $P$ represents the pressure defined by       
\begin{equation*}
P=P(h,a)=p^* h \exp{(-c(1-a))},
\end{equation*}
for given constants $p^*>0$ and $c>0$, and $\zeta$ and $\eta$ are bulk and shear viscosities.
The latter ones are given by
\begin{equation*}
    \zeta(\eps,P)=\frac{P}{2 \triangle(\eps)}  \quad \mbox{and}  \quad \eta(\eps,P)=e^{-2} \zeta(\eps,P), 
\end{equation*}
where 
\begin{equation*}
    \triangle^2(\eps) := \left(\eps_{11}^2 + \eps_{22}^2\right)\left(1+\frac{1}{e^2}\right) + \frac{4}{e^2} \eps_{12}^2 + 2 \eps_{11} \eps_{22}\left(1- \frac{1}{e^2}\right),
\end{equation*}
and $e > 1$ is the ratio of major to minor axes of the elliptical yield curve on which the principal components of the stress lie. 
The above law describes viscous-plastic material. Note that its viscosities become singular if $\triangle$ tends to zero. 
Consequently, already Hibler suggested to regularize this behaviour by bounding the viscosities when $\triangle$ is getting small and by defining maximum values $\zeta_{\mbox{\tiny max}}$ and $\eta_{\mbox{\tiny max}}$ for $\zeta$ and $\eta$. 
Then $\zeta$ and $\eta$ are given by 
\begin{equation*}
    \zeta' = \min\{\zeta,\zeta_{\mbox{\tiny max}}\} \quad \mbox{and} \quad  \eta'= \min\{\eta,\eta_{\mbox{\tiny max}}\}.
\end{equation*}
This formulation of the viscosities leads, however, to non-smooth rheology terms.

As in \cite{BDHH22,MK21}, see also \cite{KHLFG00}, we consider for $\delta > 0$ the regularization 
\begin{equation*}
    {\triangle_\delta (\eps) := \sqrt{\delta + \triangle^2 (\eps)}}.
\end{equation*}
Hence, setting $\zeta_\delta=\frac{P}{2 \triangle_\delta(\eps)}$ and $\eta_\delta=e^{-2} \zeta_\delta$, we obtain the regularized internal ice stress
\begin{equation*}
    \sigma_\delta := 2 \eta_\delta \eps + [\zeta_\delta - \eta_\delta]\trace(\eps)I - \frac{P}{2}I. 
\end{equation*}
We consider  the above momentum equation \eqref{eqmomentumbalance}  in a bounded domain $\Omega \subset \R^2$ with boundary of class $C^2$  and on a time interval $J = (0,T')$, where $0 < T' \le \infty$. 
It is coupled to two balance equations for the mean ice thickness 
\begin{equation}\label{eq:hkappa}
h\colon J \times \Omega  \to [\kappa,\infty), \quad \mbox{for } \kappa >0 \mbox{ sufficiently small}, 
\end{equation}
and the ice compactness $a\colon J \times \Omega  \to \R$ with $a \in (0,1)$, and these balance laws are given by 
\begin{equation*}\label{eq:hA}
  \left\{ \begin{array}{ll}
	\dot{h} + \divergence(u h) = S_h + d_h \Delta h, \\[2mm]
	\dot{a} + \divergence(u a) = S_a + d_a \Delta a.  \\[2mm]
	\end{array} \right.
\end{equation*} 
The assumption that $h$ is bounded from below by a small parameter $\kappa$ means that the average ice volume per control area is at least $\kappa$.
Consequently, in each control area, there is at least some sea ice.
The presence of sea ice in each control area is also reflected by the assumption that $a \in (0,1)$.
On the other hand, the assumption that $a<1$ yields that each control area is not fully covered by thick ice.

Furthermore, let $\Delta$ be the Laplacian, $d_h>0$ and $d_a>0$ be constants, and $S_h$ as well as $S_a$ represent thermodynamic terms, defined by
\begin{align}
S_h &=  f (\nicefrac{h}{a}) a + (1-a)f(0) \label{def:sh}, \\
S_a &= \begin{cases*} \frac{f(0)}{\kappa}(1-a),& if $f(0) > 0$, \\ 0, \quad & if $f(0) < 0$,
    \end{cases*}
    \quad + \quad \begin{cases*} 0,& if $S_h > 0$, \\ \frac{a}{2 h}S_h, & if $S_h < 0$.
    \end{cases*} \label{def:sa}
\end{align}
We briefly comment on these two terms:
Following \cite{Hib79}, the term $S_h$ indicates the net ice growth or melt, and it is given by the sum of the ice growth in open water or thin ice, represented by $(1-a)f(0)$, as well as the additional growth on the area covered by thick ice, and this effect is taken into account via $f (\nicefrac{h}{a}) a$.
The term $S_a$ characterizes the way in which growth and decay change the relative areal extents of thin and thick ice.
In fact, the areal fraction of thin ice decreases rapidly under freezing conditions and increases slowly under melting conditions.
For more details on thermodynamic effects in sea ice modeling, see e.g.\ \cite[Section~2]{Gol20}.

Finally, the system is completed by Dirichlet boundary conditions for $u$ and Neumann boundary conditions for $h$ and $a$.  

The complete set of equations describing sea ice dynamics by Hibler's model subject to outer forces $(g_u,g_h,g_a)$ then reads as  
\begin{equation}\label{eq:cs}
\begin{aligned}
  \left\{ \begin{array}{rll}
	m(\dot{u} + u \cdot \nabla u) &= \divergence\sigma_\delta - m \ccor n \times u - m g \nabla H + \tatm +  \tocean + g_u, & \; x\in\Omega, \ t\in J, \\[2mm]
	\dot{h} + \divergence(u h) &= S_h + d_h \Delta h + g_h, & \; x\in\Omega, \ t \in J, \\[2mm]
	\dot{a} + \divergence(u a) &= S_a + d_a \Delta a + g_a, & \; x\in\Omega, \ t \in J, \\[2mm]
	u = \frac{\partial h}{\partial\nu} &= \frac{\partial a}{\partial\nu} = 0, & \; x\in\partial\Omega, \ t\in J. \\[2mm]
	\end{array} \right.
\end{aligned}
\end{equation}
Note that $m = \rice h$, and $h$ is subject to \eqref{eq:hkappa}.

We continue  by rewriting the regularized Hibler sea ice model \eqref{eq:cs} as a quasilinear evolution equation in the Banach space $X_0$ in the periodic setting as
\begin{equation*}
  \left\{ \begin{array}{ll}
	v^\prime(t) + A(v(t))v(t) = F(v(t)), & \; \ t\in \R,\\[2mm]
	v(t) = v(t+T), & \; \ t\in \R,  \\[2mm]
	\end{array} \right.
\end{equation*}
where $v=(u,h,a)$ denotes the principal variable. Here $1<q<\infty$, and the ground space $X_0$ is defined by
\begin{equation*}
    X_0 := X_0^u \times X_0^h \times X_0^a :=  L^q(\Omega;\R^2) \times L^q(\Omega) \times L^q(\Omega).
\end{equation*}
Moreover, setting
\begin{equation*}
    W^{1,q}_0(\Omega) := \{b \in W^{1,q}(\Omega) : b=0 \mbox{ on } \partial \Omega\} \quad \mbox{and} \quad W^{2,q}_N(\Omega) := \{b \in W^{2,q}(\Omega) : \partial_\nu b=0 \mbox{ on } 
\partial \Omega\},
\end{equation*}
the regularity space $X_1$ is of the form 
\begin{equation*}
    X_1 := X_1^u \times X_1^h \times X_1^a := W^{2,q}(\Omega;\R^2) \cap W^{1,q}_0(\Omega;\R^2) \times W^{2,q}_N(\Omega) \times W^{2,q}_N(\Omega).
\end{equation*}
The quasilinear operator $A(v)$ is given by the upper triangular matrix
\begin{equation}\label{eq:op matrix}
A(v) =
    \begin{pmatrix}
    \frac{1}{\rho_{\text{ice}} h} A^H_D (\nabla u,P(h,a)) & \frac{\partial_{h} P(h,a)}{2 \rho_{\text{ice}} h}\nabla & \frac{\partial_{a} P(h,a)}{2 \rho_{\text{ice}} h}\nabla \\
    0 & - d_h \Delta_N & 0\\
    0 & 0 & - d_a \Delta_N
    \end{pmatrix},
\end{equation}
where $A^H_D$ denotes the realization of Hibler's operator subject to Dirichlet boundary conditions on $L^q(\Omega;\R^2)$ defined below in \eqref{eq:lqreal}. 
Moreover, $\Delta_N$ denotes the Neumann Laplacian on $L^q(\Omega)$ defined by $\Delta_N = \Delta$ with $D(\Delta_N) = W^{2,q}_N(\Omega)$. 

As in \cite{BDHH22}, denoting by $\eps = (\eps)_{ij}$ the deformation or rate of strain tensor, we define the map $\mathbb{S} \colon \R^{2\times 2} \to \R^{2\times 2}$ by
\begin{equation*}
    \mathbb{S} \eps
    = \begin{pmatrix}
    (1 + \frac{1}{e^2}) \eps_{11} + (1 - \frac{1}{e^2}) \eps_{22} & \frac{1}{e^2} (\eps_{12} + \eps_{21}) \\
    \frac{1}{e^2} (\eps_{12} + \eps_{21}) & (1 - \frac{1}{e^2}) \eps_{11} + (1 + \frac{1}{e^2}) \eps_{22}
    \end{pmatrix}.
\end{equation*}
Introducing 
\begin{equation*}
    S_\delta = S_\delta(\eps,P) := \frac{P}{2}\frac{\mathbb{S}\eps}{\triangle_\delta(\eps)},
\end{equation*}
we define Hibler's operator by
\begin{equation*}
    \cA^H u := - \divergence S_\delta (u).
\end{equation*}
Following \cite{BDHH22}, $\cA^H$ is given by 
\begin{equation*}
 (\cA^H u)_i
 = \sum \limits_{j,k,l=1}^2 \frac{P}{2}\frac{1}{\triangle_\delta (\eps)}\left(\mathbb{S}_{ij}^{kl} - \frac{1}{\triangle_\delta ^2 (\eps)}(\mathbb{S} \eps)_{ik} (\mathbb{S} \eps)_{jl} \right) D_k D_l u_j - 
\frac{1}{2 \triangle_\delta (\eps)}\sum \limits_{j=1}^2 (\partial_j P) (\mathbb{S} \eps)_{ij}
\end{equation*}
for $i=1,2$ and $D_m = - \mathrm{i} \partial_m$. Denote the coefficients of the principal part of $\cA^H$ by 
\begin{equation*}
{a_{ij}^{kl}(\nabla u,P) := \frac{P}{2}\frac{1}{\triangle_\delta (\eps)}\left(\mathbb{S}_{ij}^{kl} - \frac{1}{\triangle_\delta ^2 (\eps)}(\mathbb{S} \eps)_{ik} (\mathbb{S} \eps)_{jl} \right)}.
\end{equation*}
Consider sufficiently smooth fixed $v_0 = (u_0,h_0,a_0)$.
Let for instance $v_0$ be in the time trace space, i.e., $v_0 \in X_\gamma := (X_0,X_1)_{1-\nicefrac{1}{p},p} \hra B_{qp}^{2 - \nicefrac{2}{p}}(\Omega)^4 \hra C^1(\overline{\Omega})^4$, where the last embedding is valid provided $\nicefrac{1}{p} + \nicefrac{1}{q} < \nicefrac{1}{2}$.
The coefficients are then continuous, and the linearization is given by
\begin{equation*}
[\cA^H(v_0)u]_i = \sum_{j,k,l=1}^2 a_{ij}^{kl}(\nabla u_0,P(h_0,a_0))D_kD_lu_j  - \frac{1}{2 \triangle_\delta (\eps(u_0))}\sum \limits_{j=1}^2 (\partial_j P(h_0,a_0)) (\mathbb{S} \eps(u))_{ij},
\end{equation*}
see also \cite[Section~3]{BDHH22}.
The  $L^q$-realization $A^H_D(v_0)$ of $\cA^H(v_0)$ is defined by 
\begin{equation}\label{eq:lqreal}
[A^H_D(v_0)]u:= [\cA^H(v_0)]u, \quad u \in D(A^H_D(v_0)):= W^{2,q}(\Omega;\R^2) \cap W^{1,q}_0(\Omega;\R^2).
\end{equation}

The nonlinear right-hand side is given by 
\begin{equation*}
    F(v)
    := \begin{pmatrix}
   - u \cdot \nabla u - \ccor n \times u - g \nabla H + \frac{c_1}{h}|\Uatm|\Uatm +  \frac{c_2}{h}|\Uocean-u|(\Uocean -u) + g_u \\ 
    - \divergence(u h) + S_h(v) + g_h \\ 
     - \divergence(u a) + S_a(v) + g_a
    \end{pmatrix},
\end{equation*}
where $c_1 = \ratm\Catm\Ratm \rice^{-1}$, $c_2 = \rocean\Cocean\Rocean\rice^{-1}$, $\Uatm$ and $\Uocean$ are given functions, which are  interpreted as wind and oceanic currents, respectively.

Given  $T>0$, we consider solutions $v$ to Hibler's sea ice system  within the class
\begin{equation*}
    \E
    := W^{1,p}(0,T;X_0) \cap L^p(0,T;X_1).
\end{equation*}
For $v = (u,h,a) \in \E$ and fixed $v_0 = (u_0,h_0,a_0) \in X_\gamma \hra B_{qp}^{2 - \nicefrac{2}{p}}(\Omega)^4 \hra C^1(\overline{\Omega})^4$, where the last embedding is again valid if $\nicefrac{1}{p} + \nicefrac{1}{q} < \nicefrac{1}{2}$, the norm $\| v \|_{\E}$ is defined by
\begin{equation*}
    \| v \|_{\E} := \| v \|_{L^p(0,T;X_0)} + \| \Dot{v} \|_{L^p(0,T;X_0)} + \| A(v_0) v \|_{L^p(0,T;X_0)},
\end{equation*}
with $A(v_0)$ as introduced in \eqref{eq:op matrix}.

In the sequel, we suppose that a {\em strong solution} $v$ to Hibler's viscous-plastic sea ice model has the property that $v_{|(0,T)} \in \E$, and $v$ satisfies \eqref{eq:cs} for almost every $t \in (0,T)$ as equality in $L^p(0,T;X_0)$. 
The latter space then is also the underlying data space, i.e., the space in which the external forces, also referred to as data, lie, so we set
\begin{equation*}
    \F := L^p(0,T;X_0).
\end{equation*}
It is endowed with the norm
\begin{equation*}
    \| v \|_{\F} := \| v \|_{L^p(0,T;X_0)} = \left(\int_0^T \| v(t) \|_{X_0}^p \, \mathrm{d} t\right)^{\nicefrac{1}{p}}, \quad \text{with } \| \cdot \|_{X_0} :=  \| \cdot \|_{L^q(\Omega;\R^2) \times L^q(\Omega) \times L^q(\Omega)},
\end{equation*}
where $v = (u,h,a)$ denotes again the principle variable of the system.

For a densely defined sectorial operator $B$ on a Banach space $X$, the associated analytic semigroup $T(t)$ generated by $B$ as well as $p \in (1,\infty)$, the space $D_B(1-\nicefrac{1}{p},p)$ is defined by means of
\begin{equation*}
    D_B(1-\nicefrac{1}{p},p) := \left\{x \in X : [x]_{1-\nicefrac{1}{p},p} := \left(\int_0^\infty \| t^{\nicefrac{1}{p}} B T(t) x\|_{X}^p \, \frac{\mathrm{d} t}{t}\right)^{\nicefrac{1}{p}} < \infty\right\}.
\end{equation*}
Equipped with the norm $\| x \|_{1-\nicefrac{1}{p},p} := \| x \|_X + [x]_{1-\nicefrac{1}{p},p}$, the space $D_B(1-\nicefrac{1}{p},p)$ becomes a Banach space.
Then $D_B(1-\nicefrac{1}{p},p)$ coincides with the real interpolation space $(X,D(B))_{1-\nicefrac{1}{p},p}$ with equivalent norms, see e.g.\ \cite[Theorem~III.4.10.2]{Ama95} or \cite[Proposition~3.4.4]{PS16}.

The time trace space in our situation is given by
\begin{equation*}
X_{\gamma} = (X_0,X_1)_{1-\nicefrac{1}{p},p} = D_{A^H_D(v_0)}(1-\nicefrac{1}{p},p) \times D_{\Delta_N}(1-\nicefrac{1}{p},p)  \times D_{\Delta_N}(1-\nicefrac{1}{p},p)
\end{equation*}
for $1<p<\infty$. 

Throughout this article, we also employ the notation
\begin{equation*}
    \E = \E^u \times \E^h \times \E^a \quad \mbox{and} \quad \F = \F^u \times \F^h \times \F^a.
\end{equation*}
For $k \in \{u,h,a\}$, $A_u = A^H_D(v_0)$ and $A_h = A_a = \Delta_N$, the corresponding norms are given by
\begin{equation*}
    \| k \|_{\F^k} := \| k \|_{L^p(0,T;X_0^k)} \quad \text{and} \quad \| k \|_{\E^k} := \| k \|_{L^p(0,T;X_0^k)} + \| \Dot{k} \|_{L^p(0,T;X_0^k)} + \| A_k k \|_{L^p(0,T;X_0^k)}.
\end{equation*}
By Theorem~III.4.10.2 of \cite{Ama95}, we deduce that 
\begin{equation}\label{eq:embeddinginbuc}
    \E \hra \mathrm{BUC}([0,T]; X_\gamma),
\end{equation}
where $\mathrm{BUC}([0,T]; X_\gamma)$ denotes the space of all bounded and uniformly continuous functions on the time interval $[0,T]$ with values in $X_\gamma$.

We assume that $(g_u,g_h,g_a)\colon \R \to X_0$ satisfies $(g_u,g_h,g_a)_{|(0,T)} \in L^p(0,T;X_0)$ and that $g_u$, $g_h$ and $g_a$ are time periodic with period $T$. 
In addition, we assume that $g \nabla H \in L^p(0,T;X_0^u)$ is $T$-periodic.

If $c_1=c_2=0$, the right-hand side becomes
\begin{equation}\label{eq:rhsfp}
    F_p(t,v)
    = \begin{pmatrix}
    - u \cdot \nabla u - \ccor(n \times u) - g \nabla H + g_u(t) \\ - \divergence(u h) + S_h(v) + g_h(t) \\ - \divergence(u a) + S_a(v) + g_a(t)
\end{pmatrix}.
\end{equation}
We then rewrite the quasilinear Cauchy problem in the periodic setting as
\begin{equation}\label{eq:quasilinperiodacp}
  \left\{ \begin{array}{ll}
	v^\prime(t) + A(v(t)) v(t) = F_p(t,v(t)), & \; \quad t \in \R, \\[2mm]
	v(t) = v(t+T), & \; \quad t \in \R.  \\[2mm]
	\end{array} \right.
\end{equation}

Supposing  $\vert \Uatm \vert \Uatm = c(t) h$ for some periodic $\R^2$-valued function $c(\cdot)$, we may interpret $\frac{c_1}{h} \vert \Uatm \vert \Uatm$ as a  periodic function in the sea ice momentum equation, i.e.,
\begin{equation*}
    g_u(t) = c(t) c_1. 
\end{equation*}

In the following,  we denote by $\B_r(v)$ the ball in $\E$ with center $v \in \E$ and radius $r>0$.
We state now our assumptions on the term related to the force due to changing sea surface tilt $g \nabla H$, on the thermodynamic terms $S_h$ and $S_a$ as well as on the underlying ice growth rate function $f$, on the Coriolis parameter $\ccor$ and on the ice thickness and ice compactness component of the equilibrium $h_*$ and $a_*$.

\begin{assumption}{P}\label{ass:P}
Let  $T>0$ as well as $v_* = (0,h_*,a_*)$ with $h_* > 0$ and $a_* > 0$ being constant in time and space.
For $c_s > 0$, $c_f > 0$, $R>0$ and $\delta > 0$, assume that $f \in C_b^1(\R_+)$ and that
\begin{align}
    \| g \nabla H \|_{\F^u},\| S_h(v_*) \|_{\F^h}, \| S_a(v_*) \|_{\F^a} 
    &< \frac{\delta}{4},\tag{P1}\label{ass:P1}\\
    \ccor < c_s, \quad \frac{1}{2}c_s < h_*, a_* < 2 c_s \quad \mbox{as well as} \quad \normf
    &< c_f, \quad \mbox{and} \tag{P2}\label{ass:P2}\\
    \frac{1}{4}c_s < \kappa \le h_i, a_i < 4c_s \tag{P3}\label{ass:P3}
\end{align}
holds for all $v_i = (u_i,h_i,a_i) \in \overline{\B}_R(v_*)$, $i=1,2$.
\end{assumption}

\noindent
For $p,q \in (1,\infty)$ such that $\frac{1}{2} + \frac{1}{p} + \frac{1}{q} < 1$, we observe that \eqref{ass:P3} is especially fulfilled provided \eqref{ass:P2} holds true and $R>0$ is chosen small enough. 
This is valid by virtue of the embedding presented in \eqref{eq:embEintoLinftyLinfty}, see the proof of Lemma~\ref{lem:reg} for more details.
Our  main result reads then  as follows.

\begin{theorem}\label{thm:periodicsolquasilin}
Let $\Omega \subset \R^2$ be a bounded domain with boundary of class $C^2$ and $v_* = (0,h_*,a_*)$ with $h_* > 0$ and $a_* > 0$ being constant in time and space. 
Moreover, let $T>0$, and suppose $p,q \in (1,\infty)$ satisfy
\begin{equation*}
    \frac{1}{2} + \frac{1}{p} + \frac{1}{q} < 1.
\end{equation*}
Assume that $(g_u,g_h,g_a)\colon \R \to X_0$ is $T$-periodic with $(g_u,g_h,g_a)_{|(0,T)} \in L^p(0,T;X_0)$.  
Then there exist $c_s > 0$, $c_f > 0$ and $R_1 > 0$ such that for any $R \in (0,R_1)$, there exists $\delta = \delta(R) > 0$ such that if $\| \left(g_u(\cdot), g_h(\cdot), g_a(\cdot)\right)_{|(0,T)} \|_\F < \frac{\delta}{4}$ and if $g \nabla H, S_h, S_a, \ccor, v_*$ and $f$ satisfy Assumption~\ref{ass:P}, there exists a 
unique $T$-periodic strong solution $v\colon \R \to X_0$ to \eqref{eq:quasilinperiodacp} satisfying  $v_{|(0,T)} \in \overline{\B}_R(v_*)$.
\end{theorem}

For the atmospheric wind of the aforementioned form, we obtain the following corollary. 

\begin{corollary}[\bf Periodic wind forces]\label{cor:concreteshapeofgu}
Assume that in the situation of Theorem~\ref{thm:periodicsolquasilin}, the atmospheric wind $\Uatm$ is of the form $\vert \Uatm \vert \Uatm = c(t)h$, where $c\colon \R \to X_0^u$ satisfies   
$c_{|(0,T)} \in L^p(0,T;X_0^u)$ and is  $T$-periodic. 
Then there exist $c_s > 0$, $c_f > 0$ and $R_1 > 0$ such that for any $R \in (0,R_1)$, there is $\delta = \delta(R) > 0$ such that if  $\| \left(c(\cdot)c_1, g_h(\cdot), g_a(\cdot)\right)_{|(0,T)} \|_\F < \frac{\delta}{4}$ and if $g \nabla H, S_h, S_a, \ccor, v_*$ and $f$ fulfill Assumption~\ref{ass:P}, there exists a unique strong $T$-periodic 
solution $v\colon \R \to X_0$ to \eqref{eq:quasilinperiodacp} satisfying  $v_{|(0,T)} \in \overline{\B}_R(v_*)$.
\end{corollary}

Finally, we invoke the time dependence of the ice growth rate function $f$, i.e., we consider $f$ of the form $f = f(t,h(t),a(t))$ and suppose that it is $T$-periodic.
We then assume that $f \in C_b^1(J \times \R_+)$, and instead of \eqref{ass:P1} and \eqref{ass:P2}, we presume that for $c_s > 0$, $c_f > 0$, $R>0$ and $\delta > 0$, it holds that
\begin{align}
     \| g \nabla H \|_{\F^u},\| S_h(\cdot,v_*) \|_{\F^h}, \| S_a(\cdot,v_*) \|_{\F^a} 
    &< \frac{\delta}{4}, \quad \mbox{and}\tag{PT1}\label{ass:PT1}\\
    \ccor < c_s, \enspace \frac{1}{2}c_s < h_*, a_* < 2 c_s \quad \mbox{as well as} \quad \normftime
    &< c_f. \tag{PT2}\label{ass:PT2}
\end{align}
We then obtain a result on periodic solutions for Hibler's sea ice system for periodic ice growth rates.

\begin{corollary}[\bf Periodic ice growth rate]\label{cor:periodicicegrowthrate}
Assume that in the situation of Theorem~\ref{thm:periodicsolquasilin}, the ice growth rate $f$ in $S_h$ and $S_a$ given by \eqref{def:sh} and \eqref{def:sa} is time-dependent and $T$-periodic.
Then there exist $c_s > 0$, $c_f > 0$ and $R_1 > 0$ such that for each $R \in (0,R_1)$, there is $\delta = \delta(R) > 0$ such that if $\| \left(c(\cdot)c_1, g_h(\cdot), g_a(\cdot)\right)_{|(0,T)} \|_\F < \frac{\delta}{4}$ and if $g \nabla H, S_h, S_a, \ccor, v_*$ and $f$ satisfy \eqref{ass:PT1}, \eqref{ass:PT2} and \eqref{ass:P3}, there exists a unique strong $T$-periodic solution $v\colon \R \to X_0$ to \eqref{eq:quasilinperiodacp} satisfying $v_{|(0,T)} \in \overline{\B}_R(v_*)$.
In particular, the terms $S_h(\cdot,v(\cdot))$ and $S_a(\cdot,v(\cdot))$ are $T$-periodic.
\end{corollary}

\section{Maximal periodic $L^p$-regularity}\label{sec:maxperiodicreg}
For an arbitrary Banach space $X$, $1<p<\infty$ and a linear operator $A \colon D(A) \to X$, set 
\begin{equation*}
    \F = L^p(0,T;X) \quad \mbox{and} \quad \E = W^{1,p}(0,T;X) \cap L^p(0,T;D(A)).
\end{equation*}
We say that {\em $A$ admits maximal periodic $L^p$-regularity} if for every $f \in \F$, there is a unique solution $u \in \E$ to
\begin{equation*}
  \left\{ \begin{array}{ll}
	u^\prime(t) - A u(t) = f(t), \; \quad t \in (0,T), \\[2mm]
	u(0) = u(T).  \\[2mm]
	\end{array} \right.
\end{equation*}
The closed graph theorem then yields the existence of a constant $M>0$ such that
\begin{equation}\label{eq:consofmaxperiodicreg}
\| u \|_{\E} \le M \| f \|_{\F}.
\end{equation}

We will make use of the following characterization of maximal periodic $L^p$-regularity due to Arendt and Bu \cite{AB02}.
For more information on maximal regularity and its periodic version, we refer for example to \cite{Ama95, Lun95, DHP03, KW04, GK18}.

\begin{proposition}[Arendt, Bu]\label{prop:charmaxperiodicreg}
Let $X$ be a Banach space and $A \colon D(A) \to X$ be the generator of a $C_0$-semigroup on $X$. 
Then $A$ admits maximal periodic $L^p$-regularity if and only if $1 \in \rho(e^{TA})$ and $A$ admits maximal $L^p$-regularity.
\end{proposition}

For $v_* = (0,h_*,a_*)$ with $h_*>0$ and $a_*>0$ being constant in time and space and $\eps > 0$, we introduce 
\begin{equation}\label{eq:Aeps}
    A_\eps(v_*)
    := \begin{pmatrix}
    \frac{1}{\rice h_*}A_D^H(v_*) & \frac{\partial_h P_*}{2 \rice h_*}\nabla & \frac{\partial_a P_*}{2 \rice h_*}\nabla \\ 0 & - d_h \Delta_N + \eps & 0 \\ 0 & 0 & - d_a \Delta_N + \eps
\end{pmatrix}.
\end{equation}

We now show that $A_\eps(v_*)$ admits maximal periodic $L^p$-regularity. 

\begin{proposition}\label{prop:maxperiodicreghiblerop}
Let $s \in (1,\infty)$ as well as $v_* = (0,h_*,a_*)$ with $h_*>0$ and $a_*>0$ being constant in time and space. 
Then the linearized operator $A_\eps(v_*)$ admits maximal periodic $L^s$-regularity for all $\eps > 0$.
\end{proposition}

\begin{proof}

By Lemma 7.3  of \cite{BDHH22},  $A_D^H(v_*)$ is invertible on $L^r(\Omega;\R^2)$.
The invertibility of $-d_{h}\Delta_N + \eps$ and $-d_{a}\Delta_N + \eps$ on $L^r(\Omega;\R^2)$ for any $\eps >0$ and  $1<r<\infty$, $d_h>0$ and $d_a>0$, yields that $0 \in \rho(-A_\eps(v_*))$ by virtue of the upper triangular structure of $A_\eps(v_*)$, and it thus holds that $1 \in \rho(e^{-T A_\eps(v_*)})$.

Lemma~$7.3$ of  \cite{BDHH22} states that $A(v_*)$ has the property of maximal $L^s$-regularity on $L^r(\Omega;\R^2)$, $r \in (1,\infty)$.
Observing  that   $-d_{h}\Delta_N + \eps$  and $-d_{a}\Delta_N + \eps$ have maximal $L^s$-regularity, the upper triangular structure of $A_\eps(v_*)$ implies again that $A_\eps(v_*)$ has maximal $L^s$-regularity. 
The assertion follows then by Proposition~\ref{prop:charmaxperiodicreg}. 
\end{proof}

\section{Estimates for the quasilinear term and the right-hand side}\label{sec:proofprepresults}
Let $J$ be the open interval $(0,T)$. 
Taking into account the translation in the second and the third equation described in Section~\ref{sec:maxperiodicreg}, we need to adjust the right-hand side accordingly. 
To this end, we set
\begin{equation*}
    F_{\eps,p}(t,v) := \begin{pmatrix}
- u \cdot \nabla u - \ccor(n \times u) - g \nabla H + g_u(t) \\ - \divergence(u h) + S_h(v) + \eps h + g_h(t) \\ - \divergence(u a) + S_a(v) + \eps a + g_a(t)
\end{pmatrix},
\end{equation*}
so clearly $F_p = F_{\eps,p} - (0,\eps h, \eps a)$, where $F_p$ is as defined in \eqref{eq:rhsfp}.
To distinguish the latter case from the situation of $f$ being explicitly time-dependent, we define $F_{\eps,p,t}$ by
\begin{equation*}
    F_{\eps,p,t}(t,v) := \begin{pmatrix}
- u \cdot \nabla u - \ccor(n \times u) - g \nabla H + g_u(t) \\ - \divergence(u h) + S_h(t,v) + \eps h + g_h(t) \\ - \divergence(u a) + S_a(t,v) + \eps a + g_a(t)
\end{pmatrix}.
\end{equation*}

\begin{lemma}\label{lem:reg}
Let  $p,q \in (1,\infty)$ such that $\frac{1}{2} + \frac{1}{p} + \frac{1}{q} < 1$.
Assume that for $c_s>0$ and $c_f>0$ arbitrary, \eqref{ass:P2} is satisfied, let $R_0 > 0$ be small enough such that \eqref{ass:P3} holds for $v \in \overline{\B}_{R_0}(v_*)$, and let $(g_u,g_h,g_a)\colon \R \to X_0$ be $T$-periodic with $(g_u,g_h,g_a)_{|(0,T)} \in \F$.
Then $F_{\eps,p}(\cdot,v(\cdot)) \in \F$ for all $v \in \overline{\B}_{R_0}(v_*)$.

If $f$ depends explicitly on time, assume that \eqref{ass:PT2} is satisfied instead of \eqref{ass:P2}.
It then follows that $F_{\eps,p,t}(\cdot,v(\cdot)) \in \F$ for all $v \in \overline{\B}_{R_0}(v_*)$.
\end{lemma}
\begin{proof}
In the sequel, for $\beta_1$, $\beta_2 > 0$ as well as $r$, $s \in (1,\infty)$, we denote by $H^{\beta_1,s}(J;H^{\beta_2,r}(\Omega))$ the Bessel potential spaces $H^{\beta_1,s}$ on the time interval $J$ with values in the Bessel potential spaces $H^{\beta_2,r}(\Omega)$ on the bounded domain $\Omega$.
The mixed derivative theorem, see e.g.\ Cor.~$4.5.10$ in \cite{PS16}, yields for any $\theta \in (0,1)$ the embedding
\begin{equation}\label{eq:appmdthm}
    \E \hra H^{\theta,p}(J;H^{2(1-\theta),q}(\Omega))^4.
\end{equation}
In order to justify the remark on \eqref{ass:P3} above, we employ Sobolev embeddings to get
\begin{equation}\label{eq:embinterpolspaceintoLinftyLinfty}
    H^{\theta,p}(J;H^{2(1-\theta),q}(\Omega))^4 \hra L^\infty(J;L^\infty(\Omega))^4
\end{equation}
provided $\theta - \frac{1}{p} > 0$ and $2(1-\theta) - \frac{2}{q} > 0$, which is equivalent to $\frac{1}{p} < \theta < 1 - \frac{1}{q}$.
The assumption that $\frac{1}{2} + \frac{1}{p} + \frac{1}{q} < 1$ especially implies that there exists such $\theta \in (0,1)$.
Thus, combining \eqref{eq:appmdthm} and \eqref{eq:embinterpolspaceintoLinftyLinfty}, we infer that
\begin{equation}\label{eq:embEintoLinftyLinfty}
    \E \hra L^\infty(J;L^\infty(\Omega))^4.
\end{equation}

On the other hand, the assumptions on $p$ and $q$ imply that we can fix suitable $\alpha, \alpha^\prime \in (1,\infty)$ with $\frac{1}{\alpha} + \frac{1}{\alpha^\prime} = 1$ such that there are $\theta_1, \theta_2 \in (0,1)$ with
\begin{equation*}
    \theta_1 - \frac{1}{p} > -\frac{1}{\alpha p}, \quad 2(1-\theta_1) - \frac{2}{q} > 1 - \frac{2}{\alpha q}, \quad \theta_2 - \frac{1}{p} > -\frac{1}{\alpha^\prime p} \quad \mbox{and} \quad 2(1-\theta_2) - \frac{2}{q} > 1 - \frac{2}{\alpha^\prime q}.
\end{equation*}
Recalling that the spatial dimension is $2$, we deduce from Sobolev embeddings that
\begin{equation*}
    H^{\theta_1,p}(J;H^{2(1-\theta_1),q}(\Omega)) \hra L^{\alpha p}(J;W^{1,\alpha q}(\Omega)) \quad \mbox{and} \quad H^{\theta_2,p}(J;H^{2(1-\theta_2),q}(\Omega)) \hra L^{\alpha^\prime p}(J;W^{1,\alpha^\prime q}(\Omega)).
\end{equation*}
In conjunction with \eqref{eq:appmdthm}, the previous embeddings result in
\begin{equation}\label{eq:embEintoLalphapW1alphq}
    \E \hra L^{\alpha p}(J;W^{1,\alpha q}(\Omega))^4 \quad \mbox{and} \quad \E \hra L^{\alpha^\prime p}(J;W^{1,\alpha^\prime q}(\Omega))^4.
\end{equation}

To prove the assertion, let $v \in \E$. 
H\"older's inequality and \eqref{eq:embEintoLalphapW1alphq} yield 
\begin{equation*}
\begin{aligned}
    \| u \nabla u \|_{\F^u} 
    &\le \| u \|_{L^{\alpha p}(J;L^{\alpha q}(\Omega)^2)} \| \nabla u \|_{L^{\alpha^\prime p}(J;L^{\alpha^\prime q}(\Omega)^{2 \times 2})}\\
    &\le c \| u \|_{L^{\alpha p}(J;W^{1,\alpha q}(\Omega)^2)} \| u \|_{L^{\alpha^\prime p}(J;W^{1,\alpha^\prime q}(\Omega)^2)}\\
    &\le c \| v \|_{\E}^2 < \infty.
\end{aligned}
\end{equation*}

Similarly as above, and making use of the product rule, we deduce that
\begin{align*}
    \| \divergence(u h) \|_{\F^h}
    &\le \| h \divergence(u) \|_{L^p(J;L^q(\Omega))} + \| u \nabla h \|_{L^p(J;L^q(\Omega))}.
    \intertext{The treatment of the second term is analogous to the previous estimate of $u \nabla u$, while for the first term, we argue that}
    \| h \divergence(u) \|_{L^p(J;L^q(\Omega))}
    &\le \| h \|_{L^{\alpha p}(J;L^{\alpha q}(\Omega))} \| \divergence(u) \|_{L^{\alpha^\prime p}(J;L^{\alpha^\prime q}(\Omega))}\\
    &\le c \| h \|_{L^{\alpha p}(J;L^{\alpha q}(\Omega))} \| u \|_{L^{\alpha^\prime p}(J;W^{1,\alpha^\prime q}(\Omega)^2)},
\end{align*}
so in total, arguing likewise for $\divergence(u a)$, we obtain the estimates
\begin{equation*}
    \|  \divergence(u h) \|_{\F^h} \le 2 c \| v \|_{\E}^2 < \infty \quad \mbox{and} \quad 
\|  \divergence(u a) \|_{\F^a} \le 2 c \| v \|_{\E}^2 < \infty.
\end{equation*}

We next deal with the Coriolis term, for which we get, exploiting $\E \hra \F$, $\vert n \vert = 1$ as well as boundedness of the domain and the time interval,
\begin{align*}
    \| \ccor (n \times u) \|_{\F^u} 
    &\le c \ccor \| v \|_{\E} < \infty.
    \intertext{An analogous argument leads to}
    \| \eps h \|_{\F^h} 
    &\le c \eps \| v \|_{\E} < \infty \quad \mbox{and} \quad \| \eps a \|_{\F^a} \le c \eps \| v \|_{\E} < \infty.
\end{align*}

By assumption, it readily follows that $\| (g_u,g_h,g_a)_{|(0,T)} \|_{\F} < \infty$ as well as $\| g \nabla H \|_{\F^u} < \infty$. 
Presuming that $f \in C_b^1(\R_+)$, proceeding as above and additionally employing \eqref{ass:P2} and \eqref{ass:P3}, we obtain
\begin{equation*}
\begin{aligned}
    \| S_h(v) \|_{\F^h}
    &\le \| f(\nicefrac{h}{a})a \|_{L^p(J;L^q(\Omega))} + \| (1-a) f(0) \|_{L^p(J;L^q(\Omega))}\\
    &\le c(2 \normf \| v \|_{\E} + \normf) < \infty, \quad \mbox{and}\\
    \| S_a(v) \|_{\F^a}
    &\le \left\| \frac{f(0)}{\kappa}(1-a) \right\|_{\F^a} + \left\| \frac{a}{2 h} S_h(v) \right\|_{\F^a}\\
    &\le c \normf \frac{4}{c_s} (1 + \| v \|_{\E}) + 8 \| S_h(v) \|_{\F^h} < \infty
\end{aligned}
\end{equation*}
for $v \in \overline{\B}_{R}(v_*)$.

Concerning the second part of the lemma, we remark that the only difference is in the treatment of $S_h$ and $S_a$.
Making use of $\vert f(t,\cdot) \vert \le \normftime$, we use an analogous strategy as above to deduce that
\begin{equation*}
\begin{aligned}
    \| S_h(\cdot,v) \|_{\F^h} 
    &\le c(2 \normftime \| v \|_{\E} + \normftime) < \infty, \quad \mbox{and}\\
    \| S_a(\cdot,v) \|_{\F^a}
    &\le c \normftime \frac{4}{c_s}(1+\| v \|_{\E}) + 8 \| S_h(\cdot,v) \|_{\F^h} < \infty.
\end{aligned}
\end{equation*}
This completes the proof.
\end{proof}

We next show that the term $F_{\eps,p}$ on the right-hand side satisfies a Lipschitz condition.
\begin{lemma}\label{lem:lipschitzcond}
Let $p,q \in (1,\infty)$ such that $\frac{1}{2} + \frac{1}{p} + \frac{1}{q} < 1$.
Moreover, suppose that for $c_s > 0$ and $c_f > 0$ arbitrary, \eqref{ass:P2} holds true, and let $R_0>0$ be small enough such that \eqref{ass:P3} is valid for $v \in \overline{\B}_{R_0}(v_*)$.
Let also $(g_u,g_h,g_a)\colon \R \to X_0$ be $T$-periodic with $(g_u,g_h,g_a)_{|(0,T)} \in \F$.

Then there exist $C>0$ and $C_k>0$ such that for all $R \in (0,R_0)$, we have
\begin{equation}\label{eq:lipschitzcond}
    \| F_{\eps,p}(\cdot,v_1(\cdot)) - F_{\eps,p}(\cdot,v_2(\cdot)) \|_{\F}
    \le C(R+C_k) \| v_1 - v_2 \|_{\E} 
\end{equation}
for any $v_1$, $v_2 \in \overline{\B}_{R}(v_*)$.

In the situation of $f$ being explicitly time-dependent, we assume \eqref{ass:PT2} instead of \eqref{ass:P2}.
In that case, there are $C_t > 0$ and $C_{k,t} > 0$ such that for any $R \in (0,R_0)$, it holds that
\begin{equation}\label{eq:lipschitzcondtimedepf}
    \| F_{\eps,p,t}(\cdot,v_1(\cdot)) - F_{\eps,p,t}(\cdot,v_2(\cdot)) \|_{\F}
    \le C_t(R+C_{k,t}) \| v_1 - v_2 \|_{\E} 
\end{equation}
for any $v_1$, $v_2 \in \overline{\B}_{R}(v_*)$.
\end{lemma}

\begin{proof}
Let $R_0 > 0$ be as above, and consider $R \in (0,R_0)$ arbitrary as well as $v_1$, $v_2 \in \overline{\B}_R(v_*)$. 
As in the proof of Lemma~\ref{lem:reg}, we treat each term separately, and we start with $u \nabla u$. 
Then 
\begin{equation*}
\begin{aligned}
    \| - u_1 \nabla u_1 + u_2 \nabla u_2 \|_{\F^u}
    &\le \| u_1 \nabla (u_1 - u_2) \|_{L^p(J;L^q(\Omega)^2)} + \| (u_1 - u_2) \nabla u_2 \|_{L^p(J;L^q(\Omega)^2)}\\
    &\le \| u_1 \|_{L^{\alpha p}(J;L^{\alpha q}(\Omega)^2)} \| \nabla (u_1 - u_2) \|_{L^{\alpha^\prime p}(J;L^{\alpha^\prime q}(\Omega)^{2 \times 2})}\\
    &\quad + \| u_1 - u_2 \|_{L^{\alpha^\prime p}(J;L^{\alpha^\prime q}(\Omega)^2)} \| \nabla u_2 \|_{L^{\alpha p}(J;L^{\alpha q}(\Omega)^{2 \times 2})}\\
    &\le c (\| u_1 \|_{L^{\alpha p}(J;W^{1,\alpha q}(\Omega)^2)} + \| u_2 \|_{L^{\alpha p}(J;W^{1,\alpha q}(\Omega)^2)}) \| u_1 - u_2 \|_{L^{\alpha^\prime p}(J;W^{1,\alpha^\prime q}(\Omega)^2)}\\
    &\le c (\| u_1 \|_{\E^u} + \| u_2 \|_{\E^u}) \| u_1 - u_2 \|_{\E^u},\\
    &\le c R \| v_1 - v_2 \|_{\E}.
\end{aligned}
\end{equation*}

We continue with the Coriolis term, which we treat as in the proof of Lemma~\ref{lem:reg} and for which we use \eqref{ass:P2} to obtain
\begin{align*}
    \| \ccor (n \times u_1) - \ccor (n \times u_2) \|_{\F^u}
    &\le c \ccor \| v_1 - v_2 \|_{\E} < c c_s \| v_1 - v_2 \|_{\E}.
    \intertext{Similarly, we get}
    \| \eps h_1 - \eps h_2 \|_{\F^h} 
    &\le c \eps \| v_1 - v_2 \|_{\E} \quad \mbox{and} \quad \| \eps a_1 - \eps a_2 \|_{\F^a} \le c \eps \| v_1 - v_2 \|_{\E}.
\end{align*}

The next terms under consideration are $\divergence(u h)$ and $\divergence(u a)$, and we have
\begin{align*}
    \| \divergence(u_1 h_1) - \divergence(u_2 h_2) \|_{\F^h} 
    &\le \| h_1 \divergence(u_1) - h_2 \divergence(u_2) \|_{\F^h} + \| u_1 \nabla h_1 - u_2 \nabla h_2 \|_{\F^h}.
    \intertext{We take care of the two addends individually. First, analogously as in the proof of Lemma~\ref{lem:reg} and by an application of \eqref{ass:P2}, we conclude that}
    \| h_1 \divergence(u_1) - h_2 \divergence(u_2) \|_{\F^h}
    &\le \| (h_1 - h_2) \divergence(u_1) \|_{\F^h} + \| h_2 (\divergence(u_1) - \divergence(u_2)) \|_{\F^h}\\
    &\le c (2 R + h_*) \| v_1 - v_2 \|_{\E}\\
    &\le c (R + c_s) \| v_1 - v_2 \|_{\E}.
    \intertext{Likewise, we obtain}
    \| u_1 \nabla h_1 - u_2 \nabla h_2 \|_{\F^h}
    &\le \| u_1 \nabla (h_1 - h_2) \|_{\F^h} + \| (u_1 - u_2) \nabla h_2 \|_{\F^h}\\
    &\le c (R + c_s) \| v_1 - v_2 \|_{\E}.
\end{align*}
In total, using the analogy for $\divergence(u a)$, we derive that
\begin{equation*}
    \| \divergence(u_1 h_1) - \divergence(u_2 h_2) \|_{\F^h} \le c (R + c_s) \| v_1 - v_2 \|_{\E} \quad \mbox{and} \quad
\| \divergence(u_1 a_1) - \divergence(u_2 a_2) \|_{\F^a} \le c (R + c_s) \| v_1 - v_2 \|_{\E}.
\end{equation*}

The last terms to be considered are the thermodynamic terms $S_h$ and $S_a$ from \eqref{def:sh} and \eqref{def:sa}, respectively, because the periodic terms $g_u$, $g_h$ and $g_a$ as well as the term arising from the force due to changing sea surface tilt $g \nabla H$ are constant in $v$. The mean value theorem in conjunction with \eqref{ass:P2} and \eqref{ass:P3} implies
\begin{equation*}
\begin{aligned}
    \| S_h(v_1) - S_h(v_2) \|_{\F^h}
    &\le c \normf \| v_1 - v_2 \|_{\E} + \| (f(\nicefrac{h_1}{a_1}) - f(\nicefrac{h_2}{a_2})) a_1 \|_{\F^h} + \| f(\nicefrac{h_2}{a_2})(a_1 - a_2) \|_{\F^h}\\
    &\le c \normf \| v_1 - v_2 \|_{\E} + \| f(\nicefrac{h_1}{a_1}) - f(\nicefrac{h_2}{a_2}) \|_{L^{\alpha p}(J;L^{\alpha q}(\Omega))} \| a_1 \|_{L^{\alpha^\prime p}(J;L^{\alpha^\prime q}(\Omega))}\\
    &\le c \normf \left(1 + \frac{1}{c_s}(a_* + R)\right) \| v_1 - v_2 \|_{\E}\\
    &< c \left(\frac{c_f}{c_s}R + c_f\right) \| v_1 - v_2 \|_{\E}.
\end{aligned}
\end{equation*}
A similar argument exhibits that
\begin{equation*}
\begin{aligned}
    \| S_a(v_1) - S_a(v_2) \|_{\F^a}
    &\le c \normf \frac{1}{c_s} \| v_1 - v_2 \|_{\E} + \left\| \frac{a_1}{2 h_1} (S_h(v_1) - S_h(v_2)) \right\|_{\F^a} + \left\| \left(\frac{a_1}{2 h_1} - \frac{a_2}{2 h_2}\right) S_h(v_2) \right\|_{\F^a}\\
    &\le c \normf \left(\frac{1}{c_s}R + \frac{1}{c_s} + 1\right) \| v_1 - v_2 \|_{\E} + c \frac{1}{c_s} \| S_h(v_2) \|_{\F^h} \| v_1 - v_2 \|_{\E}\\
    &\le c \normf \left(\frac{1}{c_s}R + \frac{1}{c_s} + 1\right)\| v_1 - v_2 \|_{\E}\\
    &< c \left(\frac{c_f}{c_s}R + \frac{c_f}{c_s} + c_f\right)\| v_1 - v_2 \|_{\E},
\end{aligned}
\end{equation*}
showing that \eqref{eq:lipschitzcond} is valid.

Regarding the second part of the lemma, we employ a similar argument as in the proof of Lemma~\ref{lem:reg}, i.e., we estimate $\vert f(t,\cdot) \vert$ and $\nabla_x f(t,x) \vert$ by $\normftime$.
We then infer that
\begin{equation*}
\begin{aligned}
    \| S_h(\cdot,v_1) - S_h(\cdot,v_2) \|_{\F^h}
    &\le c \left(\frac{c_f}{c_s}R + c_f\right) \| v_1 - v_2 \|_{\E}, \quad \mbox{as well as}\\
    \| S_a(\cdot,v_1) - S_a(\cdot,v_2) \|_{\F^a}
    &\le c \left(\frac{c_f}{c_s}R + \frac{c_f}{c_s} + c_f\right)\| v_1 - v_2 \|_{\E}.
\end{aligned}
\end{equation*}
This shows that \eqref{eq:lipschitzcondtimedepf} is also valid.
\end{proof}

For later use, we concatenate the previous estimates and thus determine $C$ and $C_k$ from the statement of Lemma~\ref{lem:lipschitzcond} precisely.

\begin{remark}\label{rem:shapeofconstants}
The above estimates show that 
\begin{equation*}
    \| F_{\eps,p}(\cdot,v_1(\cdot)) - F_{\eps,p}(\cdot,v_2(\cdot)) \|_\F \le C(R + C_k) \| v_1 - v_2 \|_{\E},
\end{equation*}
where
\begin{equation}\label{eq:shapeofC}
    C = c\left(1 + \frac{c_f}{c_s}\right),
\end{equation}
and
\begin{equation}\label{eq:shapeofCk}
    C_k
    = \left(\eps + c_s + \left(1+\frac{1}{c_s}\right)c_f\right) \frac{1}{1 + \frac{c_f}{c_s}}.
\end{equation}
We remark that also $C_t$ and $C_{k,t}$ take a similar form as $C$ and $C_k$, respectively.
\end{remark}

The last preparatory result concerns the quasilinear matrix $A_\eps$.

\begin{lemma}\label{lem:lipschitzcondmatrix}
Let $p,q \in (1,\infty)$ be such that $\frac{1}{2} + \frac{1}{p} + \frac{1}{q} < 1$.
Additionally, assume that for $c_s > 0$ and $c_f > 0$ arbitrary, \eqref{ass:P2}, or, in the situation of $f$ being explicitly time-dependent, \eqref{ass:PT2}, holds true, and let $R_0$ be sufficiently small such that \eqref{ass:P3} is satisfied for $v \in \overline{\B}_{R_0}(v_*)$.
Moreover, let $V \subset X_\gamma$ open be such that for $v \in V$, it holds that $h \ge \kappa$ and $a > 0$.
Then $A_\eps \colon V \to \mathcal{L}(X_1,X_0)$ is a family of closed linear operators, and for any $R \in (0,R_0)$, there is $L(R) > 0$ such that
\begin{equation}\label{eq:lipschitzcondquasilinmatrix}
    \| A_\eps(v_1(\cdot))w(\cdot) - A_\eps(v_2(\cdot))w(\cdot) \|_\F 
    \le L(R) \| v_1 - v_2 \|_{\E} \| w \|_{\E}
\end{equation}
for all $v_1$, $v_2 \in \overline{\B}_R(v_*)$ and for any  $w \in X_1$.
\end{lemma}

\begin{proof}
The fact that $A_\eps\colon V \to \mathcal{L}(X_1,X_0)$ is a family of closed linear operators follows from results in \cite{BDHH22}. 
Using a similar argument as in Lemma~$6.2$ of \cite{BDHH22} involving that $h \ge \frac{c_s}{4}$, we find that for every $R \in (0,R_0)$ there is $l(R) > 0$ such that
\begin{equation}\label{eq:A1analogon}
    \| A_\eps(v_1)w - A_\eps(v_2)w \|_{X_0} 
    \le l(R) \| v_1 - v_2 \|_{X_\gamma} \| w \|_{X_1}
\end{equation}
for $v_1$, $v_2 \in \overline{\B}_R(v_*)$ and $w \in X_1$.

Arguing as in Section~$3$ of \cite{HS20}, we exploit \eqref{eq:A1analogon} and \eqref{eq:embeddinginbuc} to get
\begin{equation*}
\begin{aligned}
    \| A_\eps(v_1(\cdot))w(\cdot) - A_\eps(v_2(\cdot))w(\cdot) \|_{\F} 
    &\le l(R) \| v_1 - v_2 \|_{\mathrm{BUC}([0,T];X_\gamma)} \| w \|_{\E}\\
    &\le l(R) C_{\E} \| v_1 - v_2 \|_{\E} \| w \|_{\E},
\end{aligned}
\end{equation*}
where $C_{\E}$ denotes the embedding constant in \eqref{eq:embeddinginbuc}. 
Setting $L(R) := l(R) C_{\E}$ finally yields \eqref{eq:lipschitzcondquasilinmatrix}.
\end{proof}

\section{Proof of the main results}\label{sec:proofmainresults}

The proof of Theorem~\ref{thm:periodicsolquasilin} relies on an application of the contraction mapping principle after rewriting the quasilinear time periodic problem as a linearized problem with nonlinear right-hand side.
We then use the notion of maximal periodic $L^p$-regularity as explained in Section~\ref{sec:maxperiodicreg} as well as the Lipschitz estimates established in Section~\ref{sec:proofprepresults}.

\begin{proof}[Proof of Theorem~\ref{thm:periodicsolquasilin}]
Let $R \in (0,R_0)$ for $R_0 > 0$ as in Lemmas~\ref{lem:reg}, \ref{lem:lipschitzcond} and \ref{lem:lipschitzcondmatrix}. For $v \in \overline{\B}_R(v_*)$ and $A_\eps(v_*)$ as in \eqref{eq:Aeps}, we define the map $\Phi \colon \overline{\B}_R(v_*) \to \overline{\B}_R(v_*)$ by $\Phi(v) := w$, with $w$ being the unique solution to
\begin{equation*}
  \left\{ \begin{array}{ll}
	w^\prime(t) + A_\eps(v_*) w(t) = F_{\eps,p}(t,v(t)) + \left(A_\eps(v_*) - A_\eps(v(t))\right)v(t), \; \quad t \in (0,T), \\[2mm]
	w(0) = w(T).  \\[2mm]
	\end{array} \right.
\end{equation*}

In order to apply the contraction mapping principle, we verify that $\Phi$ is a selfmap and a contraction. To this end, let  $v$, $\vbar \in \overline{\B}_R(v_*)$. 
First, we note that $v_*$ is the unique solution to
\begin{equation*}
  \left\{ \begin{array}{ll}
	w_*^\prime(t) + A_\eps(v_*) w_*(t) = A_\eps(v_*)v_*, \; \quad t \in (0,T), \\[2mm]
	w_*(0) = w_*(T),  \\[2mm]
	\end{array} \right.
\end{equation*}
since  $v_*^\prime(t) + A_\eps(v_*)v_* = A_\eps(v_*)v_*$, which follows from $v_*$ being especially constant in time. By linearity, $\Phi(v) - v_*$ is the unique solution to
\begin{equation*}
  \left\{ \begin{array}{ll}
	\Tilde{w}^\prime(t) + A_\eps(v_*) \Tilde{w}(t) = F_{\eps,p}(t,v(t)) + \left(A_\eps(v_*) - A_\eps(v(t))\right)v(t) - A_\eps(v_*)v_*, \; \quad t \in (0,T), \\[2mm]
	\Tilde{w}(0) = \Tilde{w}(T).  \\[2mm]
	\end{array} \right.
\end{equation*}
We deduce from the latter observation, from maximal periodic $L^p$-regularity of $A_\eps(v_*)$ by Proposition~\ref{prop:maxperiodicreghiblerop}, from Lemma~\ref{lem:reg} and from \eqref{eq:consofmaxperiodicreg} in the first step and from Lemmas~\ref{lem:lipschitzcond} and \ref{lem:lipschitzcondmatrix} as well as the identity $F_{\eps,p}(t,v_*) - A_\eps(v(t))v_* = F_p(t,v_*)$ in the second step that
\begin{equation}\label{eq:selfmap}
    \begin{aligned}
    \| \Phi(v) - v_* \|_{\E}
    &\le M \| F_{\eps,p}(\cdot,v(\cdot)) - F_{\eps,p}(\cdot,v_*) \|_{\F} + M \| F_{\eps,p}(\cdot,v_*) - A_\eps(v(\cdot))v_* \|_{\F}\\
    & \quad+ M \| \left(A_\eps(v_*) - A_\eps(v(\cdot)) \right)(v(\cdot) - v_*) \|_{\F}\\
    &\le M C (R + C_k) \| v - v_* \|_{\E} + M \| F_{p}(\cdot,v_*) \|_{\F} + M L(R) \| v - v_* \|_{\E}^2, 
    \end{aligned}
\end{equation}
where $C$ and $C_k$ are as in \eqref{eq:shapeofC} and \eqref{eq:shapeofCk}, respectively. 

Likewise, and additionally making use of the fact that $A_\eps(v(t))v_* = A_\eps(\vbar(t))v_* = A_\eps(v_*)v_*$, we infer that
\begin{equation}\label{eq:contraction}
    \begin{aligned}
    \| \Phi(v) - \Phi(\vbar) \|_{\E}
    &\le M \| F_{\eps,p}(\cdot,v(\cdot)) - F_{\eps,p}(\cdot,\vbar(\cdot)) \|_{\F} + M \| \left(A_\eps(v_*) - A_\eps(v(\cdot))\right)(v(\cdot) - \vbar(\cdot)) \|_{\F}\\
    & \quad \quad + M \| \left(A_\eps(v(\cdot)) - A_\eps(\vbar(\cdot))\right)(\vbar(\cdot) - v_*) \|_{\F}\\
    &\le M C (R + C_k) \| v - \vbar \|_{\E} + M L(R) \| v - v_* \|_{\E} \| v - \vbar \|_{\E}\\
    & \quad \quad + M L(R) \| v - \vbar \|_{\E} \| \vbar - v_* \|_{\E}\\
    &\le \left(M C (R + C_k) + 2 M L(R) R\right) \| v - \vbar \|_{\E}. 
    \end{aligned}
\end{equation}

The representation of $C_k$ given in \eqref{eq:shapeofCk} reveals that there are $c_s > 0$ and $c_f > 0$ with the following property: 
If we assume \eqref{ass:P2} and \eqref{ass:P3}, and if we consider $\eps > 0$ sufficiently small, then  $C_k \le \frac{1}{4 M C}$. 
More precisely, choosing $\eps < \frac{1}{12 M c}$ and taking into account $c_s \le \frac{1}{12 M c}$ as well as $c_f < \frac{c_s}{12 M (c_s + 1)c}$, we derive that
\begin{equation*}
    C_k = \left(\eps + c_s + \left(1+\frac{1}{c_s}\right)c_f\right) \frac{1}{1 + \frac{c_f}{c_s}} < \left(\frac{1}{12 Mc} + \frac{1}{12 Mc} + \frac{1}{12 Mc}\right) \frac{1}{1 + \frac{c_f}{c_s}} = \frac{1}{4 M c(1+\frac{c_f}{c_s})} = \frac{1}{4 M C}.
\end{equation*}

We additionally consider $R \le \min\left(R_0, \frac{1}{4 M C}, \frac{1}{8 M L(R)}\right)$. Moreover, the choice $\delta = \delta(R) = \frac{R}{4 M}$ yields that
\begin{equation*}
\begin{aligned}
    \| F_{p}(\cdot,v_*) \|_\F 
    &= \| \left(- g \nabla H + g_u(\cdot), S_h(v_*) + g_h(\cdot), S_a(v_*) + g_a(\cdot)\right) \|_\F\\
    &\le \| g \nabla H \|_{\F^u} + \| S_h(v_*) \|_{\F^h} + \| S_a(v_*) \|_{\F^a} + \| \left(g_u(\cdot), g_h(\cdot), g_a(\cdot)\right)_{|(0,T)} \|_\F < \delta = \frac{R}{4 M}
\end{aligned}
\end{equation*}
by the smallness assumption on $(g_u,g_h,g_a)$ in the statement of the theorem and in view of \eqref{ass:P1}.

Inserting the smallness of $C_k$, $R$ and $\delta$ in \eqref{eq:selfmap} and \eqref{eq:contraction}, we then conclude that
\begin{align*}
    \| \Phi(v) - v_* \|_{\E} 
    &\le \frac{1}{2} \| v - v_* \|_{\E} + \frac{1}{4}R + \frac{1}{8} \| v - v_* \|_{\E} \le \frac{7}{8}R
\intertext{and}
    \| \Phi(v) - \Phi(\vbar) \|_{\E} 
    &\le \left(\frac{1}{2} + \frac{1}{4}\right) \| v - \vbar \|_{\E} \le \frac{3}{4} \| v - \vbar \|_{\E}.
\end{align*}

Therefore, $\Phi\colon \overline{\B}_R(v_*) \to \overline{\B}_R(v_*)$ is a selfmap and a contraction, so we may exploit the contraction mapping theorem to get a unique solution $v$ on the period $(0,T)$, satisfying $v(0) = v(T)$. 
We may thus extend $v$ periodically to the whole real line and thereby obtain the $T$-periodic solution to \eqref{eq:quasilinperiodacp}.
\end{proof}

Corollary~\ref{cor:concreteshapeofgu} is a direct application of Theorem~\ref{thm:periodicsolquasilin}, so it remains to show Corollary~\ref{cor:periodicicegrowthrate}.

\begin{proof}[Proof of Corollary~\ref{cor:periodicicegrowthrate}]
Let again $R \in (0,R_0)$ for $R_0 > 0$ as in Lemmas~\ref{lem:reg}, \ref{lem:lipschitzcond} and \ref{lem:lipschitzcondmatrix}.
Given $v \in \overline{\B}_R(v_*)$, we define the map $\Phi_t\colon \overline{\B}_R(v_*) \to \overline{\B}_R(v_*)$ by $\Phi_t(v) := w$, where $w$ is the unique solution to
\begin{equation*}
  \left\{ \begin{array}{ll}
	w^\prime(t) + A_\eps(v_*) w(t) = F_{\eps,p,t}(t,v(t)) + \left(A_\eps(v_*) - A_\eps(v(t))\right)v(t), \; \quad t \in (0,T), \\[2mm]
	w(0) = w(T),  \\[2mm]
	\end{array} \right.
\end{equation*}
and $A_\eps(v_*)$ as in \eqref{eq:Aeps}.

We then proceed as in the proof of Theorem~\ref{thm:periodicsolquasilin}, this time employing the respective second part in Lemmas~\ref{lem:reg} and \ref{lem:lipschitzcond} and making use of similar identities as in the aforementioned proof.
As in the proof of Theorem~\ref{thm:periodicsolquasilin}, for $c_s > 0$ and $c_f > 0$ sufficiently small, we consider $R \le \min\left(R_0,\frac{1}{4 M C},\frac{1}{8 M L(R)}\right)$ and choose $\delta = \frac{R}{4 M}$.
The smallness of $(g_u,g_h,g_a)$ as well as the assumptions \eqref{ass:PT1}, \eqref{ass:PT2} and \eqref{ass:P3} then yield the estimates
\begin{equation*}
    \| \Phi_t(v) - v_* \|_{\E} \le \frac{7}{8}R \quad \mbox{and} \quad \| \Phi_t(v) - \Phi_t(\vbar) \|_{\E} \le \frac{3}{4} \| v - \vbar \|_{\E}.
\end{equation*}
Applying the contraction mapping principle to $\Phi_t$, we obtain a unique solution $v$ on the period $(0,T)$, which fulfills $v(0) = v(T)$.
The $T$-periodic solution to the original Cauchy problem then is created by extending the previous $v$ to the whole real line.

The structure of the thermodynamic terms $S_h$ and $S_a$ in the concrete case exhibits that they are $T$-periodic if the time-dependent ice growth rate function $f$ is $T$-periodic, which is valid by assumption, and if the variables $h$ and $a$ are $T$-periodic.
This is ensured by the fact that the solution $v$ discussed above is $T$-periodic by construction.
\end{proof}


\begin{thebibliography}{99}

\bibitem{Ama95}
H.~Amann,
{\it Linear and Quasilinear Parabolic Problems}, Monographs in Mathematics, vol.~89, Birkh\"auser, 1995.

\bibitem{AB02}
W.~Arendt, S.~Bu,
The operator-valued Marcinkiewicz multiplier theorem and maximal regularity.
{\it Math. Z.} {\bf 240} (2002), 311--343.

\bibitem{BDHH22}
F.~Brandt, K.~Disser, R.~Haller-Dintelmann, M.~Hieber,
Rigorous Analysis and Dynamics of Hibler's Sea Ice Model. 
{\it J. Nonlinear Sci.} {\bf 32} (2022), Paper No.~50. 
 
\bibitem{DWT15} S.~Danilov, Q.~Wang, R.~Timmermann, N.~Iakovlev, D.~Sidorenko, M.~Kimmritz, T.~Jung,  J.~Schr\"{o}ter,
Finite-Element Sea Ice Model (FESIM). 
{\it Geosci. Model Dev.} {\bf 8} (2015), 1747--1761.

\bibitem{DHP03} 
R.~Denk, M.~Hieber, J.~Pr\"uss,
{\it $\mathcal{R}$-Boundedness, Fourier Multipliers and Problems of Elliptic and Parabolic Type}. 
Mem. Amer. Math. Soc. {\bf 788} 2003.

\bibitem{Fel08}
D.L.~Feltham, 
Sea Ice Rheology.
{\it Annu. Rev. Fluid Mech.} {\bf 40} (2008), 91--112.

\bibitem{Gal13} 
G.P.~Galdi,
On time periodic flow of a viscous liquid past a moving cylinder.
{\it Arch. Ration. Mech. Anal.} {\bf 210} (2013), 451-498. 

\bibitem{GK18}
G.P.~Galdi, M.~Kyed,
Time periodic solutions to the Navier-Stokes equations.
In: {\it Handbook of Mathematical Analysis in Mechanics of Viscous Fluids}, Y.~Giga, A.~Novotny (eds.) 509--578, Springer, 2018. 

\bibitem{GS04} 
G.P.~Galdi, H.~Sohr, 
Existence and uniqueness of time-periodic physically reasonable Navier-Stokes flows past a body. 
{\it Arch. Ration. Mech. Anal.} {\bf 172} (2004), 363--406. 

\bibitem{GHH16}
M.~Geissert, M.~Hieber, T.H.~Nguyen,
A general approach to time periodic incompressible viscous flow problems. 
{\it Arch. Ration. Mech. Anal.} {\bf 220} (2016), 1095--1118.  

\bibitem{Gol15}
K.~Golden,
The mathematics of sea ice. 
In: {\it The Princeton Companion to Applied Mathematics}, N.~Highan (ed.), 694--705, Princeton University Press, Princeton, 2015.

\bibitem{Gol20} K.~Golden, L.~Bennetts, E.~Cherkaev, I.~Eiseman, D.~Feltham, C.~Horvat, E.~Hunke, C.~Jones, D.~Perovich, P.~Ponte-Castaneda, C.~Strong, D.~Sulsky, A.~Wells,
\newblock Modeling sea ice. 
\newblock {\em Notices Amer. Math. Soc.} {\bf 67} (10) (2020), 1535-1564. 

\bibitem{Hib79}
W.D.~Hibler,
A Dynamic Thermodynamic Sea Ice Model.
{\it J. Phys. Oceanogr.} {\bf 9} (1979), 815--846.

\bibitem{HS20}
M.~Hieber, C.~Stinner,
Strong time periodic solutions to Keller-Segel systems: An approach by the quasilinear Arendt-Bu theorem.
{\it J. Differential Equations} {\bf 269} (2020), 1636--1655.

\bibitem{KDL15} 
M.~Kimmrich, S.~Danilov, M.~Lorsch,
On the convergence of the modified elastic-viscous-plastic method for solving the sea ice momentum equation.
{\it J. Comput. Phys.} {\bf 296} (2015), 90--100.

\bibitem{KMT14}
H.~Kozono, Y.~Mashilo, R.~Takada,
Existrence of periodic solutions and their asymptotic stability to the Navier-Stokes equations with Coriolis force. 
{\it J. Evol. Equ.} {\bf 14} (2014), 565--601.

\bibitem{KHLFG00}
M.~Kreyscher, M.~Harder, P.~Lemke, G.~Flato, M.~Gregory,
Results of the Sea Ice Model Intercomparison Project: Evaluation of sea ice rheology schemes for use in climate simulations.
{\it J. Geophys. Res.} {\bf 105} (2000), 11299--11320.

\bibitem{KW04} 
P.~Kunstmann, L.~Weis, 
{Maximal $L_p$-regularity for Parabolic Equations, Fourier Multiplier Theorems and $\mathcal{H}^{\infty}$-functional Calculus}.
In: {\it Functional Analytic Methods for Evolution Equations}. M. Iannelli, R. Nagel and S. Piazzera (eds.), Springer, 2004, 65--311. 

\bibitem{LT09} J.-F.~Lemieux, B.~Tremblay,
Numerical convergence of viscous-plastic sea ice models.
{\it J. Geophys. Res.} {\bf 114} (2009), C05009.

\bibitem{LTT21} X.~Liu, M.~Thomas, E.S.~Titi,
Well-Posedness of Hibler's Dynamical Sea-Ice Model. 
{\it J. Nonlinear Sci.} {\bf 32} (2022), Paper No.~49. 

\bibitem{Lun95}
A.~Lunardi,
{\it Analytic Semigroups and Optimal Regularity in Parabolic Problems}, In: H.~Brezis et al. (eds.) Progress in Nonlinear Differential Equations and their Applications, vol.~16, Birkh\"auser Verlag, Basel, 1995.

\bibitem{Meh19}
C.~Mehlmann,
{\it Efficient numerical methods to solve the viscous-plastic sea ice model at high spatial resolutions}, PhD thesis, Otto-von-Guericke Universit\"at Magdeburg, 2019.

\bibitem{Mehalle21}
C.~Mehlmann, S.~Danilov, M.~Losch, J.F.~Lemieux, N.~Hutter, T.~Richter, P.~Blain, E.C.~Hunke, P.~Korn,
Simulating linear kinematic features in viscous-plastic sea ice models on quadrilateral and triangular grids with different variable staggering. 
{\it J. Adv. Model. Earth Syst.} {\bf 13} (2021), e2021MS002523.

\bibitem{MK21} 
C.~Mehlmann, P.~Korn,
Sea-ice on triangular grids.
{\it J. Comput. Phys.} {\bf 428} (2021), 110086.

\bibitem{MR17} C.~Mehlmann, T.~Richter, 
A modified global Newton solver for viscous-plastic sea ice models.
{\it Ocean Model.} {\bf 116} (2017), 96--117.

\bibitem{PT22} P.~Piersanti, R.~Temam,
On the dynamics of grounded shallow ice sheets: Modelling and analysis. 
{\it Adv. Nonlinear Anal.} {\bf 12} (2023), 20220280.
 
\bibitem{PS16} J.~Pr\"uss, G.~Simonett,
{\it Moving Interfaces and Quasilinear Parabolic Evolution Equations}. Monographs in Mathematics, vol. 105, Birkh\"auser, 2016.

\bibitem{SK18}
C.~Seinen, B.~Khouider, 
Improving the Jacobian free Newton-Krylov method for the viscous-plastic sea ice momentum equation.
{\it Physica D} {\bf 376--377} (2018), 78--93.


\end{thebibliography}
\end{document}